\documentclass[a4paper]{article}
\usepackage{amsthm,amssymb,amsmath,enumerate,graphicx,epsf,mathbbol,stmaryrd}

\newcommand{\COLORON}{1}
\newcommand{\NOTESON}{0}
\newcommand{\Debug}{1}
\usepackage[usenames,dvips]{color} 
\usepackage{amsthm,amssymb,amsmath,enumerate,graphicx,epsf,mathbbol,stmaryrd}
\usepackage[dvips, bookmarks, colorlinks=false]{hyperref}

\newcommand{\comment}[1]{}
\newcommand{\COMMENT}[1]{}

\definecolor{darkgray}{rgb}{0.3,0.3,0.3}
\newcommand{\defi}[1]{{\color{darkgray}\emph{#1}}}

\comment{
	\begin{lemma}\label{}	
\end{lemma}
\begin{proof}

\end{proof}

\begin{theorem}\label{}
\end{theorem} 
\begin{proof} 	

\end{proof}

}

\newtheorem{proposition}{Proposition}[section]
\newtheorem{definition}[proposition]{Definition}
\newtheorem{theorem}[proposition]{Theorem}
\newtheorem{corollary}[proposition]{Corollary}
\newtheorem{lemma}[proposition]{Lemma}

\newtheorem{conjecture}{{\color{red}Conjecture}}[section]

\newtheorem{problem}[conjecture]{{\color{red}Problem}}

\newtheorem{examp}[proposition]{Example}%[section]

\newcommand{\kreis}[1]{\mathaccent"7017\relax #1}

\newcommand{\FIG}{0}

\ifnum \NOTESON = 1 \newcommand{\note}[1]{ 

	\ 

	{\color{blue} \hspace*{-60pt} NOTE: \color{Turquoise}{\small  \tt \begin{minipage}[c]{1.1\textwidth}  #1 \end{minipage} \ignorespacesafterend }} 
	
	\ 
	
	}
\else \newcommand{\note}[1]{} \fi

\ifnum \Debug = 1 
\else  \fi

\ifnum \FIG = 1 \newcommand{\fig}[1]{Figure ``{#1}''}
\else \newcommand{\fig}[1]{Figure~\ref{#1}} \fi

\ifnum \Debug = 1 \usepackage[notref,notcite]{showkeys}
\fi

\ifnum \COLORON = 0 \renewcommand{\color}[1]{}
\fi

\newcommand{\showFig}[2]{
   \begin{figure}[htbp]
   \centering
   \noindent
   \epsfbox{#1.eps}
   \caption{\small #2}
   \label{#1}
   \end{figure}
}

\newcommand{\N}{\ensuremath{\mathbb N}}
\newcommand{\R}{\ensuremath{\mathbb R}}
\newcommand{\Z}{\ensuremath{\mathbb Z}}

\newcommand{\cq}{\mathcal Q}

\newcommand{\oo}{\ensuremath{\omega}}

\newcommand{\eps}{\ensuremath{\epsilon}}

\newcommand{\sig}{\ensuremath{\sigma}}

\newcommand{\fcg}{\ensuremath{|G|}}

\newcommand{\sm}{\backslash}
\newcommand{\restr}{\upharpoonright}

\newcommand{\nin}{\ensuremath{{n\in\N}}}

\newcommand{\unin}{\ensuremath{[0,1]}}

\newcommand{\limf}[1]{\ensuremath{\lim \inf(#1)}}

\newcommand{\sgl}[1]{\ensuremath{\{#1\}}}
\newcommand{\pth}[2]{\ensuremath{#1}\text{--}\ensuremath{#2}~path}

\newcommand{\pths}[2]{\ensuremath{#1}\text{--}\ensuremath{#2}~paths}
\newcommand{\arc}[2]{\ensuremath{#1}\text{--}\ensuremath{#2}~arc}

\newcommand{\seq}[1]{\ensuremath{(#1_i)_{i\in\N}}} 
\newcommand{\fml}[1]{\ensuremath{\{#1_n\}_{n\in \N}}} 
\newcommand{\ffml}[2]{\ensuremath{\{#1_n\}_{n\in #2}}} % family with given index set
\newcommand{\flo}[2]{\ensuremath{#1}\text{--}\ensuremath{#2}~flow} % x-y flow
\newcommand{\flos}[2]{\ensuremath{#1}\text{--}\ensuremath{#2}~flows} % x-y flows

\newcommand{\g}{\ensuremath{G\ }}
\newcommand{\G}{\ensuremath{G}}

\newcommand{\s}{s}

\newcommand{\Lr}[1]{Lemma~\ref{#1}}
\newcommand{\Tr}[1]{Theorem~\ref{#1}}
\newcommand{\Sr}[1]{Section~\ref{#1}}
\newcommand{\Srs}[2]{Sections~\ref{#1} to~\ref{#2}}

\newcommand{\Cr}[1]{Corollary~\ref{#1}}
\newcommand{\Cnr}[1]{Con\-jecture~\ref{#1}}

\newcommand{\Dr}[1]{De\-fi\-nition~\ref{#1}}

\newcommand{\lf}{locally finite}
\newcommand{\lfg}{locally finite graph}

\newcommand{\nlf}{non-locally-finite}
\newcommand{\nlfg}{non-locally-finite graph}

\renewcommand{\iff}{if and only if}
\newcommand{\fe}{for every}

\newcommand{\st}{such that}

\newcommand{\ti}{there is}
\newcommand{\tho}{there holds}

\newcommand{\obda}{without loss of generality}

\newcommand{\wrt}{with respect to}

\newcommand{\inm}{infinitely many}

\newcommand{\FC}{Freudenthal compactification}

\newcommand{\tocir}{topological circle}
\newcommand{\topa}{topological path}

\newcommand{\labtequ}[2]{ \begin{equation} \label{#1} 	\begin{minipage}[c]{0.9\textwidth}  #2 \end{minipage} \ignorespacesafterend \end{equation} }

\newcommand{\mySection}[2]{}

\newcommand{\DK}{Diestel and K\"uhn}

\newcommand{\CDB}{\cite[Section~8.5]{diestelBook05}}
\newcommand{\unit}{\mathbb 1}
\newcommand{\gid}[1]{\overline{#1}}

\newcommand{\ltopf}[1]{\ensuremath{#1\text{-}TOP}}
\newcommand{\ltop}{\ensuremath{\ltopf{\ell}}}
\newcommand{\ltopx}[1]{\ensuremath{ |#1|_\ell }}

\newcommand{\ltp}{\ltopx{$G$}}
\newcommand{\ltpf}[1]{\ensuremath{|G|_{#1}}}
\newcommand{\blg}{\ensuremath{\partial^\ell G}}

\newcommand{\lER}{\ensuremath{\ell: E(G) \to \R^+_*}}

\newcommand{\eti}{\ensuremath{|G|}}

\newcommand{\Getop}{\etopg}

\newcommand{\eqT}{\approx}

\newcommand{\ltpn}{\ensuremath{\ltpr{n}}}
\newcommand{\ltpr}[1]{\ensuremath{\ltp \sm \kreis{E_{#1}}}}
\newcommand{\ltprv}[1]{\ensuremath{\ltp^{#1}}}
\newcommand{\ltpnv}{\ensuremath{\ltprv{n}}}

\renewcommand{\Getop}{\ensuremath{||G||}}

\newcommand{\are}{\vec{e}}
\newcommand{\arE}{\vec{E}}

\newcommand{\arf}{\vec{f}}
\renewcommand{\arg}{\vec{g}}

\newcommand{\bp}{boundary point}

\newcommand{\finl}{\ensuremath{\sum_{e \in E(G)} \ell(e) < \infty}}

\newcommand{\kfl}{Kirchhoff's node law}
\newcommand{\ksl}{Kirchhoff's cycle law}
\newcommand{\cutr}{non-elusive} %cut respecting}
\newcommand{\Cutr}{Non-elusive} %cut respecting}
\newcommand{\dbst}[1]{\ensuremath{#1 ^{**}}}
\newcommand{\refl}{relaxed flow}

\title{Uniqueness of electrical currents in a network of finite total resistance\footnote{Published in {\em Journal of the L.M.S.} Vol.\ 82 Nr.\ 1.} }%The {D}irichlet problem in a network of finite total resistance}
\author{Agelos Georgakopoulos\thanks{Supported by a grant of the German-Israeli Foundation and by  FWF grant P-19115-N18. This work was conceived and partly written when the author was a postdoc at the University of Hamburg.} \medskip \\
 {Technische Universit\"at Graz}\\
  {Steyrergasse 30, 8010}\\
  {Graz, Austria}\\
  \\
  \small Mathematics Subject Classification: 05C21, 05C80
}

\date{}
\begin{document}
\maketitle

\noindent

\begin{abstract}
We show that if the sum of the resistances of an electrical network $N$ is finite, then there is a unique electrical current in $N$ provided we do not allow, in a sense 
%to be
made precise in the paper, any flow to escape to infinity.
\end{abstract}

\section{Introduction}

Electrical networks are physical objects but also useful tools in mathematics. For example, they are  closely related to random walks \cite{LyonsBook,WoessBook09},
they find applications in the study of Riemannian manifolds \cite{KanaiRough2,KanaiRough,CThyperb}, and
they are related to various problems in combinatorics \cite{biggs}. An electrical network $N$ has an underlying graph \g and a function $r: E(G) \to \R^+$ assigning resistances to the edges of \G. If \g is finite, then the electrical current in $N$ ---between two fixed vertices $p,q$ and with fixed intensity $I$--- is the unique flow satisfying \ksl, which demands that the potential differences sum to zero along every cycle of the graph. Recall that a flow in a graph by definition satisfies \kfl, which demands that current is preserved at every vertex other than $p$ and $q$. See \Sr{sdefs} for more precise statements of these laws. If \g is infinite then several such flows may exist, and one of the standard problems in the study of infinite electrical networks is to specify under what conditions such a flow is unique, see e.g.\ \cite{woessCurrents,thomInfNet}. 

Our main result is that if the sum of all resistances in a network $N$ is finite, then there is a unique electrical current in $N$, provided we do no allow any flow to escape to infinity; more precisely, we require that for every finite edge-cut $F$ of $G$ that does not separate the source $p$ from the sink $q$ the net flow through $F$ is zero. We call a flow satisfying this condition  \defi{\cutr}. To see the necessity of this requirement consider the network of \fig{draynet}.
%\vspace*{-0.2cm}
\showFig{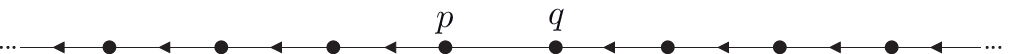}{A flow escaping to infinity}
This network admits several flows, all of which satisfy \ksl\ as there are no cycles: one of these flows runs only along the edge $pq$;  another can be obtained by sending a flow from $p$ all the way to the left end, and collecting the same amount of flow from the right end into $q$ as depicted in \fig{draynet}. We consider the latter flow to be rather pathological, since it is permitted even if we disconnect the graph by removing the edge $pq$. And indeed, it is not a \cutr\ flow. The interested reader will experiment with more complicated networks, and convince himself that requiring a flow to be \cutr\ is a natural way of preventing a flow from using infinity in an abusive manner. It is worth remarking that in an 1-ended graph every flow is \cutr. \Cutr\ flows allowed a generalisation of the well-known Max-Flow Min-Cut theorem to infinite networks \cite{mfmc} (they were called finite-cut-respecting flows in the latter paper).

We can now state our main result:

\begin{theorem}\label{finr}
Let $N=(G,r,p,q,I)$ be a \lf\ network with $\sum_{e\in E} r(e)<\infty$. Then there is a unique \cutr\ \flo{p}{q} with intensity $I$ and finite energy in $N$ that satisfies \ksl.
\end{theorem} 
% *** ---- *** 

The energy of a flow $f$ is defined by $W(f):=\sum_{e\in E(G)} f^2(e) r(e)$. The requirement that the energy be finite is necessary in the above statement (see \Sr{suniq}), and very common in the literature.

The essential part of \Tr{finr} is the uniqueness rather than the existence of the desired flow. Let us briefly consider the case when $G$ is finite. To prove uniqueness in that case, suppose there are two flows with the required properties, and consider their difference $z$. Then $z$ is a circulation in $G$, and so there must be some cycle along which $z$ is always positive in the same direction. But such a cycle yields a contradiction to the fact that $z$ must satisfy \ksl\ being the difference of two flows that do. Now back to the case when $G$ is infinite, the above argument breaks down as an infinite circulation need not traverse any finite cycle in the same direction. However, it is possible to prove that if a circulation is \cutr, then it must traverse some (finite or infinite), \defi{topological circle}; that is, a homeomorphic image of the real unit circle $S^1$ in the \defi{end-compactification $|G|$} of $G$. Such circles, introduced by \DK\ \cite{cyclesI}, have been the object of intense study recently \cite[Section~8.5]{diestelBook05}, and some of the acquired machinery is used here; see \Sr{defEnds} for more about circles. As an intermediate result we obtain that these circles must also satisfy \ksl\ if all finite cycles do.

Our main result also relies on some recent results from \cite{ltop} about  topologies on graphs induced by an assignment of lengths to the edges; see \Sr{defsLtop} for details.

The proof of \Tr{finr} spans \Srs{sint}{suniq}. In \Sr{scond} we discuss the complementary case of finite total conductance. In \Sr{NLF} we extend to \nlfg s. Finally, in \Sr{stoch} we discuss the relation of our results to stochastic processes and offer a conjecture related to the Dirichlet problem.

\section{Definitions and basic facts} \label{sdefs}

We will use the terminology of Diestel \cite{diestelBook05} for graph theoretical terms and the terminology of \cite{armstrong} for topological ones.

A network is a tuple $N=(G,r,p,q,I)$, where $G$ is an (undirected) {(mu\-lti--) graph}, 
\comment{taking directed graphs will be cumbersome if you allow multiple edges; Thomassen and Lyons also use undirected graphs.}
$r$ is a mapping assigning a \defi{resistance} $r(e)\in \R^+$ to each edge $e$ of \G, $p,q\in V(G)$, and $I\in \R$ is a constant. A \defi{flow} in $N$ is a real-valued $\flo{p}{q}$ in \g with \defi{intensity} $I$ (intuitively $p$ and $q$ are connected to a battery generating a constant current of intensity $I$). A flow $f$ satisfies, by definition, \kfl\ at every vertex except $p$ and $q$, and the net flow leaving $p$ is the \defi{intensity} of $f$. More formally, we have

\medskip

{\bf \kfl:}
For every vertex $x\in V(G)$ there holds
\begin{equation} \label{kI}
\sum_{xy\in E(G)} f(x,y) = \begin{cases}
0& \text{ if $x \neq p,q$,}\\
I& \text{ if $x= p$,}\\
-I& \text{ if $x=q$}.
\end{cases}
\tag{\ensuremath{K1}} \end{equation}
(Recall that $f$ is a function from $V^2$ to $\R$ satisfying $f(x,y)=-f(y,x)$ \fe\ $xy\in E(G)$.) A \defi{circulation} is a flow of intensity 0.

A flow $f$ in $N$ is called \defi{\cutr} if for every finite cut $(X,X')$ such that both $p,q$ lie in $X$ (or both lie in $X'$) there holds $f(X,X')=0$, where $f(X,X')=\sum_{xy\in E, x\in X, y\in X'} f(x,y)$. It follows  that if $p\in X$ and $q \in X'$ then $f(X,X')= f(\sgl{p}, V(G) - p) = - f(\sgl{q}, V(G) - q)= I$. Note that if $G$ is finite or 1-ended (see \Sr{defEnds} for the definition of an end) then every flow is \cutr.

%For a graph $G$ and $p,q\in V(G)$, a \defi{$\flo{p}{q}$} $f$ of value $I$ is an assignment of a direction and a value to each edge of $G$ that satisfies Kirchhoff's first law:

If $G$ is finite, then one defines the \defi{electrical current} $i$ to be the flow in $N$ satisfying \ksl, and it is well-known ---and not hard to prove--- that this flow always exists and that it is unique.

A \defi{directed cycle}  is a cycle together with a choice of one of its two possible orientations. If $C$ is a directed cycle then we let \defi{$\arE(C)$} denote the set of ordered pairs $(x,y)$ such that $xy$ is an edge of $C$ traversed from $x$ to $y$ in the chosen orientation of $C$.

\medskip

{\bf \ksl:}  
For every %closed walk $C=x_1 e_1 x_2 e_2 \ldots e_{k-1} x_k=x_1$
directed cycle $C$ in $G$ there holds 
\begin{equation} \label{KII} \begin{minipage}[c]{0.85\textwidth}
$\sum_{\are \in \arE(C)} v(\are) = 0$
%$\sum_{1\leq j < k} v(x_j,x_{j+1}) =0$ %where $i_*(e_j)= i(e_j)$ if $e_j$ is directed from $x_j$ to $x_{j+1}$ and $i_*(e_j):= -i(e_j)$ otherwise.
\end{minipage}\ignorespacesafterend \tag{\ensuremath{K2}} \end{equation}
where $v(\are):= i(\are) r(\are)$ is the \defi{voltage drop} or \defi{potential difference} induced by $i$ along $e$ (in physics the equation $v(\are)= i(\are) r(\are)$  is known as \defi{Ohm's law}).

An important concept for both finite and infinite electrical networks is that of energy: the \defi{energy of the flow $f$} is defined by $W(f):= \sum_{xy\in E(G)} f^2(x,y) r(e)$. ($W(f)$ is usually called ``energy'' in mathematics, but in physics it is called ``power''.)

The following theorem is well known, and indicates the importance of the concept of energy for electrical networks. See \Sr{secExist} for a proof.

\begin{theorem} \label{exist}
Let $N$ be a finite network and let $W=W(I,p,q)$ be the infimum of $W(f)$ over all flows $f$ in $N$. Then, there is a unique flow $i$ in $N$ satisfying $W(i)=W$. This flow satisfies \ksl.
\end{theorem}

\subsection{Ends, the \FC\ and wild circles in graphs} \label{defEnds}

Let \g be a graph, fixed throughout this section.

A $1$-way infinite path is
called a \defi{ray}, a $2$-way infinite path is
a \defi{double ray}. A \defi{tail} of a ray $R$ is an infinite (co-final) subpath of $R$.
Two rays $R,L$ in $G$ are \defi{equivalent} if no finite set of edges
separates them. The corresponding equivalence
classes of rays are the \defi{end\s} of $G$. We
denote the set of end\s\ of $G$ by \defi{$\Omega =
\Omega(G)$}. %A ray belonging to the end\ \oo\ is an \defi{\oo-ray}. 

We now endow the space consisting of \G, considered as a 1-complex, and its ends with the topology \eti. Firstly, every edge $e\in E(G)$ inherits the open sets corresponding to open sets of $[0, 1]$. Moreover, for every finite edge-set $S \subset E(G)$, 
we declare all sets of the form 
\begin{equation}
C(S, \oo) \cup \Omega(S,\oo) \cup
E'(S,\oo)
\label{eq}
\end{equation}
to be open; here, $C(S, \oo)$ is any component of $G - S$ and $\Omega(S,\oo)$ denotes the set of
all ends of $G$ having a ray in $C(S, \oo)$ and $E'(S,\oo)$ is any union of
half-edges $(z, y]$, one for every edge
$e = xy$ in $S$ with $y$ lying in $C(S, \oo)$. Let \Getop\ denote the topological space of $G \cup \Omega$ endowed with the topology generated by
the above open sets. Moreover, let \defi{\eti} denote the space obtained from \Getop\ by identifying any two points that have the same open neighbourhoods. (Our notation is slightly non-standard: our ``ends'' are usually called edge-ends in the literature, and the symbol $|G|$ often denotes a different space if \g is \nlf. However, for a \lf\ \g these differences disappear and our notation agrees with the mainstream.) If a point $x$ of \eti\ resulted from the identification of a vertex with some other points (possibly also vertices), then, with slight abuse of notation, we still call $x$ a \defi{vertex}. It is easy to see that two vertices $v,w$ of \g are identified in \eti\ \iff\ there are \inm\ edge-disjoint \pths{v}{w}. 

It it well-known (see \cite{ends}) that \fcg\ coincides with the \defi{Freudenthal compactification}
\cite{Freudenthal31} of the 1-complex $G$ if \g is \lf. 

The study of \fcg, in particular of topological circles therein, has been a very active field recently. It has been demonstrated by the work of several authors (\cite{locFinTutte,partition,LocFinMacLane,degree,cyclesI,cyclesII,hp,fleisch,geo,hcs}) that many well known results about paths and cycles in finite graphs can be generalised to \lf\ ones if the classical concepts of path and cycle are interpreted topologically, i.e.\ replaced by the concepts of a (topological) arc and circle in \fcg; see \Sr{defsTop} for some definitions and \CDB\ for an exposition of this field. An example of such a \tocir\ is formed by a double ray both rays of which converge to the same end together with that end. There can however be much more exciting circles in \fcg: in \fig{wild}, the $\aleph_0$ many thick double rays together with the continuum many ends of the graph combine to form a single \tocir\ $W$, the so-called \defi{wild} circle. The double rays are arranged within $W$ like the rational numbers within the reals: between any two there is a third one; see \cite{cyclesI} for a more precise description of $W$.

\showFig{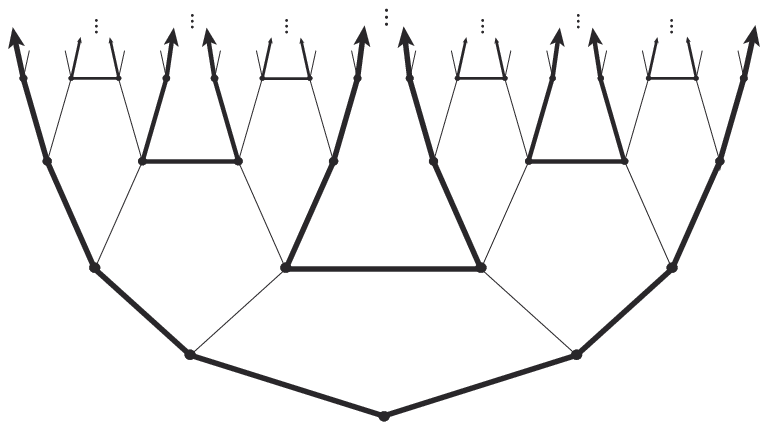}{The `wild' circle of \DK, formed by infinitely many (thick) double rays and con\-tinuum many ends.}

For a circle $C$ we let $E(C)$ denote the set of edges it traverses. We define directed circles and the notation $\arE(C)$ similarly to the case when $C$ is a finite cycle.

\subsection{Topological paths, circles, etc.} \label{defsTop}

A \defi{circle} in a topological space $X$ is a homeomorphic copy of the unit circle $S^1$ of $\R^2$ in $X$.  An \defi{arc} $R$ in 
$X$ is a homeomorphic image of the real interval $[0, 1]$ in $X$. Its \defi{endpoints} are the images of $0$ and $1$ under any homeomorphism from \unin\ to $R$. If $x,y\in R$ then $xRy$ denotes the subarc of $R$ with endpoints $x,y$. A \defi{topological path} in $X$ is a continuous map from a closed real interval to $X$. 

\newcommand{\elll}{{\it l}}

Let $\sigma: [a,b]\to X$ be a topological path in a metric space $(X,d)$. For a finite sequence $S=s_1, s_2, \ldots, s_k$ of points in $[a,b]$, let $\ell(S):= \sum_{1\leq i< k} d(\sigma(s_i),\sigma(s_{i+1}))$, and define the \defi{length} of $\sigma$ to be $\elll(\sigma):=\sup_S \ell(S)$, where the supremum ranges over 
all finite sequences $S=s_1, s_2, \ldots, s_k$ with $a=s_1<s_2< \ldots <s_k=b$. If $C$ is an arc or a circle in $(X,d)$, then we define its length $\elll(C)$ to be the length of a surjective topological path $\sigma:[0,1] \to C$ that is injective on $(0,1)$; it is easy to see that $\elll(C)$ does not depend on the choice of \sig.

\subsection{\ltop} \label{defsLtop}

%Every graph \g in this paper is considered to be a 1-compex, which means that the edges of $G$ are homeomorphic copies of the real unit interval. 

Fix a graph $G$ and a function \lER. This naturally gives rise to a distance function $d_\ell$ between the points of \G, and we let \ltp, also called \ltop, denote the corresponding metric space. 

To make this more precise, for each edge $e\in E(G)$ fix a homeomorphism $\sigma_e$ from $e$ to the real interval $[0,\ell(e)]$; by means of $\sigma_e$, any \defi{half-edge} $f$, i.e.\ any connected subset of an edge,  with endpoints $a,b$ obtains a length $\ell(f)$, namely $\ell(f):=|\sigma_e(a)-\sigma_e(b)|$. Now use $\ell$ to define a distance function on $G$: for any $x,y \in V(G)$ let $d_\ell(x,y)= \inf_{P \text{ is an \pth{x}{y}}} \ell(P)$, where $\ell(P):= \sum_{e \in E(P)} \ell(e)$. For points $x,y \in G$ that might lie in the interior of an edge we define $d_\ell(x,y)$ similarly, but instead of graph-theoretical paths we consider arcs in the 1-complex $G$: let $d_\ell(x,y)= \inf_{P \text{ is an \arc{x}{y}}} \left (\sum_{f \text{ is an edge or half-edge in $P$} } \ell(f) \right )$. By identifying any two vertices $x,x'$ of $G$ for which $d_\ell(x,x')=0$ holds we obtain a metric space $(\gid{G}, d_\ell)$. Note that if \g is \lf\ then $\gid{G}=G$. Let $\ltp$ be the completion of $(\gid{G}, d_\ell)$.

The \defi{\bp s} of $G$ are the elements of the set $\blg:=\ltp \sm \pi(\gid{G})$, where $\pi$ is the canonical embedding of $\gid{G}$ in its completion \ltp.

\note{For a subspace $X$ of $\ltp$ we write $E(X)$ for the set of edges contained in $X$, and we write $\kreis{E}(X)$ for the set of maximal half-edges contained in $X$; note that $\kreis{E}(X) \supseteq E(X)$. Similarly, for a topological path $\tau$ we write $E(\tau)$ for the set of edges contained in the image of $\tau$.}

%In this paper we will often encounter special cases of \ltp\ that induce some other well known topology on some space $G'$ containing $G$, e.g.\ the \FC\ of $G$. In order to be able to formally state the fact that the two topologies are the same, we introduce the following notation. Let $X,X'$ be  topological spaces that contain another topological space $G$. We will write $X \eqT_G X'$, or simply $X \eqT X'$ if $G$ is fixed, if the identity on $G$ extends to a homeomorphism between $G'$ and $G''$.

The space \ltp\ was introduced in \cite{ltop}, where several important special cases were found, and many basic facts were proved. Here we list some of these facts that we are going to use in this paper. For the first of them, it is easy to realize the connection to the current paper by interpreting the resistance of an edge as its length:

\begin{theorem}[\cite{ltop}] \label{finl}
If  $\sum_{e \in E(G)} \ell(e) < \infty$ then $\ltp \eqT \fcg$.
\end{theorem}
(Where ``$\eqT$'' means that the identity on \g extends to a homeomorphism between the two spaces.)

For our next lemma, fix an enumeration $e_0,e_1,\ldots$ of $E(G)$, and let $E_n:=\{e_0,\ldots, e_n\}$. Moreover, let $\kreis{e_n}$ denote the set of inner points of the edge $e_n$, and let $\kreis{E}_n:= \bigcup \{\kreis{e_0},\ldots, \kreis{e_n}\}$.

\begin{lemma}[\cite{ltop}] \label{epsNLF}
Let $C$ be a circle or arc in \ltp\ such that $E(C)$ is dense in $C$. Then, for every $\epsilon\in \R^+$ there is an $n\in\N$ such that for every subarc of $C$ in $\ltpn$ connecting two vertices $v,w$ there holds  $d_\ell(v,w)<\epsilon$.
\end{lemma}

\begin{lemma}[\cite{ltop}] \label{lcisle}
If $\sum_{e\in E(G)} \ell(e)<\infty$ then for every circle or arc $C$ in \ltp\  there holds $\elll(C)=\sum_{e\in E(C)} \ell(e)$.
\end{lemma}

\section{Intermediate results} \label{sint}

\subsection{Splitting infinite circles into cycles}

In this section we show that it is possible to write any infinite circle in a graph as a sum of a sparse family of finite cycles. We will later use this fact to show that the infinite circles of a network must also satisfy \eqref{KII} if all finite ones do, which will play an important role in the proof of our main result in \Sr{suniq}.

Call a family of edge sets \defi{sparse} if no edge appears in more than three members of the family.

\begin{lemma}\label{sparse}
Let \g be a countable graph and let ${C}$ be a directed circle in \eti. Then there is a sparse family \ffml{D}{\N}\ of finite directed cycles $D_n$ in \g \st\ $\arE(C) = \sum_n \arE(D_n)$,
\end{lemma}
where the latter sum is a formal sum of directed edges with coefficients in \Z\ in which two edges cancel out if they correspond to the same edge with opposite orientations.
(\Lr{sparse} was proved implicitely also in \cite{kirch2}.)
% *** ---- *** 
\begin{proof}

Let $e_0, e_1, \ldots$ be an enumeration of $E(G)$ and let $\arf_0, \arf_1, \ldots$ be an enumeration of $\arE(C)$. Pick an assignment $\lER$ \st\ \finl, and recall that by \Tr{finl} \tho\ $\ltp \eqT \eti$.

We will construct \ffml{D}{\N}\ recursively, in \oo\ steps. To begin with, Let $D_0$ be any finite cycle containing $f_0$, directed in the direction of $\arf_0$. Combined with $C$ this cycle gives rise to what we call a \defi{quasi-circle}: that is, a continuous image of $S^1$ comprising a finite path (in this case $D_0 \sm f_0$), called the \defi{green} part of the quasi-circle, and a subarc of $C$ (in this case $C \sm f_0$), called its \defi{blue} part. Note that the green and blue part of a quasi-circle might have some common vertices and edges. We now construct the other $D_i$ inductively, in \oo\ steps. For $i=1,2,\ldots$, suppose we have already constructed $D_0, \ldots D_{i-1}$, and specified a set $\cq_{i-1}$ of quasi-circles (let $\cq_0$ be the singleton containing the above quasi-circle) so that 
\begin{enumerate}
\item \label{si} \fe\ $\arf\in \arE(C)$, either $\arf\in\sum_{j<i} \arE(D_j)$ (in the right direction) or $f$ is contained in the blue part of some quasi-circle in $\cq_{i-1}$;

\item \label{sii}  for any two distinct elements $L,M$ of $\cq_{i-1}$, the blue parts of $M$ and $L$ are disjoint and the green parts of $M$ and $L$ are edge-disjoint.

\item \label{siii} the green part of every element of $\cq_{i-1}$ is contained in $\bigcup_{j<i} D_j$.
\end{enumerate}

Then, in step $i$, let $k=k(i)$ be the least index \st\ $\vec{f}_k \not\in \sum_{j<i} \arE(D_j)$ (we choose this $k$ even if $\overleftarrow{f_k} \in \sum_{j<i} \arE(D_j)$). By \ref{si} there is an $M=M(i) \in \cq_{i-1}$ the blue part of which contains $f_k$, and by \ref{sii} this $M$ is unique.

Let $h= h(i):= \min \{\ell(e) \mid e\in \bigcup_{j<i} E(D_j) \cup \{f_k\} \}$. By our choice of $\ell$ and \Lr{finl}, we can apply \Lr{epsNLF} to $C$, which yields an $n=n(i)$ \st\  for every subarc of $C$ in $\ltpn$ connecting two vertices $v,w$ there holds  $d_\ell(v,w)<h$. As the blue part $M_b$ of $M$ is, by construction, a subarc of $C$, the latter assertion also holds for $M_b$. Let $x,y$ be the endpoints of $M_b$, and let $M_b^x$ (respectively, $M_b^y$) be the component of $M_b \sm \kreis{f_k}$ containing $x$ (resp.\ $y$). We may assume \obda\ that 
\labtequ{obd}{$h< d(M_b^x, M_b^y)/2$,}
where $d(M_b^x, M_b^y)$ is the minimum distance of a point in $M_b^x$ from a point in $M_b^y$, for otherwise we could have chosen a smaller bound than $h$ before applying \Lr{epsNLF}.

Our aim now is to construct the directed trail $D_i$ that contains $\vec{f}_k$ as well as the green part $M_g$ of $M$, the latter traversed in the opposite direction. To achieve this, we need to construct two paths $P_x, P_y$, edge-disjoint from each other and from $M_g$, each path joining an endvertex of $f_k$ to an endvertex of $M_g$. 

To construct $P_x$, let $\arg_1, \ldots, \arg_r$ be an enumeration of the directed edges in $E_n \cap E(M_b^x)$ in the order and orientation they appear on $M_b^x$ as we move from $f_k$ to $x$. Let $P_x^1$ be a path from the endvertex of $f_k$ in $M_b^x$ to the tail of $\arg_1$, and \fe\ $1<j\leq r$ let $P_x^j$ be a path from the head of $\arg_j$ to the tail of $\arg_{j+1}$. Finally, let $P_x^{r+1}$ be a path from the head of $\arg_r$ to $x$ (some of these paths may be trivial).  We can now combine these paths with the edges $g_1, \ldots, g_r$ to obtain a directed $f_k$-$x$~walk, and shortcutting this walk if necessary we can transform it into a path, which path we call $P_x$. Note that shortcutting a walk does not influence the order and direction in which the remaining edges are traversed. We construct $P_y$ similarly. Note that by our choice of $E_n$, we could have chosen the paths $P_x^j$ so that $\ell(P_x^j) <h$ \fe\ $j$. This, and the choice of $h$, implies that no $P_x^j$ can contain an edge in $\bigcup_{j<i} D_j$ as any such edge is longer than $P_x^j$, from which we obtain that
\labtequ{Px}{$P_x \cup P_y$ contains no edge in $(\bigcup_{j<i} D_j) \sm (E_n\cap E(C))$.} Moreover, by \eqref{obd} we may assume that $P_x$ and $P_y$ are disjoint; indeed, if some $P_x^j$ has a vertex in common with some $P_y^{j'}$, then their union $P_x^j\cup P_y^{j'}$ contains an $M_b^x$-$M_b^y$~path of length at most $\ell(P_x^j) + \ell(P_y^{j'})< 2h< d(M_b^x, M_b^y)$, a contradiction.

Let $D_i:= \vec{f}_k P_x (-M_g) (-P_y)$, where $(-P)$ denotes the path $P$ traversed in the inverse direction. Note that $D_i$ is not necessarily a cycle, as we would like it to be, but rather a closed walk, but we will later modify it into a sum of cycles. 
%By \eqref{Px} and the construction of $D_n$ we have
%\labtequ{disj}{???$D_n$ contains no edge in $(\bigcup_{j<i} E(D_j))\sm E(C)$.}
By the construction of $D_i$ we have
\labtequ{back}{$D_i$ traverses $M_g$, and does so in the opposite direction as $M$ does.} %Moreover, \fe\ edge $e\in E(C)$}
Moreover, $D_i$ traverses no edge in $E(G)\sm E(C)$ more than once and traverses each edge in $E(C)$ at most twice (the latter can occur for an edge that happens to lie in $E_n \cap E(C) \cap M_g$).

To  complete step $i$, it remains to define $\cq_i$. To obtain $\cq_i$ from $\cq_{i-1}$, we remove $M$ and add the quasi-circles obtained as follows.
Consider the subspace $M_b^x \sm (E(P_x) \cap \{\kreis{g_1}, \ldots, \kreis{g_r} \})$ of $M_b^x$; note that for every component $K$ of this subspace there is a subpath $P_K$ of $P_x$ connecting the endvertices of $K$. Now $K \cup P_K$ defines a quasi-circle whose blue part is $K$ and whose green part is $P_K$. Similarly for $M_b^x$ and $P_y$. Add all these quasi-circles to $\cq_{i-1} \sm M$ to obtain $\cq_i$. It follows from our choice of $P_K$, \ref{siii} and \eqref{Px} that 
\labtequ{thin}{no edge appears in the green part of more than one element of $\cq:= \bigcup \cq_i$.}
Moreover, it is easy to check that 
\labtequ{union}{$D_i$ is the edge-disjoint union of edges in $\{g_1, \ldots, g_r \} \subset E(C)$ and green parts of elements of $\cq_i$.}
Finally, it is easy to see that \ref{si}, \ref{sii} and \ref{siii} are all satisfied by $\cq_i$ if they were satisfied by $\cq_{i-1}$ (which is the case by our inductive hypothesis).

%, and we consider this quasi-circle $Q$ directed so that 
%\labtequ{q}{$Q$ traverses its blue part $K$ in the same direction as $C$ does.}

We have thus constructed the family \ffml{D}{\N}, and by \eqref{thin} this family has the desired property that no edge is traversed more than three times (in fact, an edge in $E(G) \sm E(C)$ is traversed either twice (in opposite directions) or not at all, and an edge in $E(C)$ is traversed either once or three times). Note that by our choice of the edge $f_k$, every edge of $E(C)$ will eventually appear in the sum $\sum_{j<i} \arE(D_j)$, and in fact with the right orientation, for some step $i$. It also follows from our choice of the edge $f_k$ that every quasi-circle in $\cq$ will be considered as $M(i)$ for some step $i$, and so by \eqref{back} its green part will eventually disappear from the sum of the $D_n$. It follows from these observations and \eqref{union} that $\sum_\nin \arE(D_n) = \vec{E}(C)$.

Thus, the family \ffml{D}{\N}\ has all the desired properties except that $D_n$  is not necessarily a cycle but rather a closed walk that might traverse some edges twice and visit vertices more often. This, however, is easy to amend: if $D_n$ visits some vertex $x$ more than once, then we can split it into closed subwalks that each visit $x$ only once, while traversing each edge in the same direction as $D_n$ does. Performing this operation recursively, we can split $D_n$ into a family $D'_1, \ldots D'_k$ of directed cycles \st\ $\sum \vec{E}(D'_j) = \vec{E}(D_n)$, and so replacing each $D_n$ by such a family we obtain the desired result.
\end{proof}

Interestingly, in the last proof we used \ltop\ to prove an assertion that at first sight does not seem to be related to it. I would be interested to see a graph-theoretical proof of \Lr{sparse}.

\Lr{sparse} motivates the following problem.

\begin{problem}
 Let \g be a \lfg, and let $C$ be a circle in $\fcg$. Prove that $G$ has a planar subgraph $H$ containing $E(C)$ \st\ $E(C)$ is also  the edge-set of a circle in $|H|$ (so that $H$ can be drawn in the plane with $E(C)$ bounding a face).
\end{problem}

\comment{
\begin{equation} \label{KIIp} \begin{minipage}[c]{0.85\textwidth}
$\sum_{e\in \arE(C)} v(\are) =0$ %where $i_*(e_j)= i(e_j)$ if $e_j$ is directed from $x_j$ to $x_{j+1}$ and $i_*(e_j):= -i(e_j)$ otherwise.
\end{minipage}\ignorespacesafterend \tag{\ensuremath{K'2}} \end{equation}
}

\subsection{Shortcutting topological paths}

It is a well-known and useful fact that the image of a \topa\ in a Hausdorff space contains an arc with the same endpoints \cite{ElemTop}. For our uniqueness proof in \Sr{suniq} we will need a stronger version of this fact, saying, intuitively, that this arc can be chosen so that its points are traversed in the same order as in the original \topa. In order to state this more formally, let $\sig, \tau: \unin\to X$ be topological paths in a  Hausdorff space $X$. We say that $\tau$ \defi{shortcuts} $\sig$ if there is a monotone increasing injection $m : \unin \to \unin$ such that $\tau(x)=\sig \circ m(x)$ for every $x\in \unin$. We will prove that
\begin{lemma}\label{oriarc}
Let $\sigma: \unin \to X$ be a topological path in a Hausdorff space $X$, and suppose its endpoints $u:= \sigma(0), v:=\sigma(1)$ are distinct. Then, there is an injective topological \pth{u}{v}\ in $X$ that shortcuts \sig.
\end{lemma} 
% *** ---- *** 

For this we are going to need the following
\begin{lemma}[{\cite[Lemmas 5.12, 5.13]{ElemTop}}] \label{hall}
Let $X$ be a Hausdorff space, and let $\sig: \unin \to X$ be a topological path in $X$. Then, there exists a closed subset $F$ of $\unin$ such that
\begin{enumerate}
\item \label{hi}   $\sig(0), \sig(1) \in \sig(F)$;
\item \label{hii}  If $C$ is any component of $\unin \sm F$ and $x,y$ are the end points of $C$, then $\sig(x)=\sig(y)$;
\item No proper closed subset of $F$ satisfies both~\ref{hi} and~\ref{hii}, and
\item \label{hiv} $\sig(F)$ is an arc.
\end{enumerate}
\end{lemma} 

We now proceed with the proof of \Lr{oriarc}.

\begin{proof}
Apply \Lr{hall} to \sig\ to obtain a subset $F$ of $\unin$ with properties~\ref{hi}-\ref{hiv}. By~\ref{hiv} $A:=\sig(F)$ is an arc, so let $b: A \to \unin$ be a homeomorphism. Define the metric $d_A$ on $A$ by $d_A(x,y):= |b(x)-b(y)|$; clearly, $d_A$ is compatible with the topology of $A$. 

We now define a mapping $\sig': \unin \to X$ that will be an intermediate step towards the construction of $\tau$. For every point $x\in F$ let $\sig'(x)= \sig(x)$. For every component $C$ of $\unin\sm F$, recall that by \ref{hii} $\sig(x)=\sig(y)=p$ where $x,y$ are the end points of $C$, and let $\sig(z):=p$ for every $z\in C$. Obviously, $\sig'$ is a topological path. Moreover, it is easy to see that $\sig'$ shortcuts \sig; indeed, just define $m:\unin\to \unin$ to map any point $x\in F$ to itself and any point $z\not\in F$ to the endpoint of the component of $\unin\sm F$ in which it lies.

Note that $\sig'$ is \defi{almost injective}, that is, the preimage of any point is an interval of $\unin$. This implies that the length $h$ of $\sig'$ with respect to $d_A$ is finite, since, easily, any almost injective path in $\unin$ has finite length. Thus we may define a mapping $\tau: \unin \to \sig(F)$ by mapping any point $x\in \unin$ to a point $\sig'(y)$ such that the restriction $\sig' \restr [0,y]$ of $\sig'$ to $[0,y]$ has length $x/h$ (with respect to $d_A$). Note that by the definition of the length of a topological path $\tau$ is well defined, i.e.\ if $y,y'\in \unin$ are such that $\elll(\sig' \restr [0,y])=\elll(\sig' \restr [0,y'])$ then $\sig'(y)=\sig'(y')$. It is straightforward to check that $\tau$ shortcuts $\sig'$, and thus $\tau$ also shortcuts $\sig$ since $\sig'$ shortcuts $\sig$. Note that $\tau(0)=u$ and $\tau(1)=v$. 

It is also easy to prove that $\tau$ is injective; for if $x\leq y \in \unin$ then $\sig'$ contains a topological path $\sig''$ from $\tau(x)$ to $\tau(y)$ such that $\elll(\sig'')=(y-x)/h>0$, and as $\sig''$ is almost transitive this means that its endpoints $\tau(x),\tau(y)$ are distinct.

We claim next that $\tau$ is continuous. For this, pick a point $x\in\unin$ and let $O$ be an open ball of radius $\eps\in \R^+$ with respect to $d_A$ around $\tau(x)$. We have to show that $\unin$ has an open set $U \ni x$ such that $\tau(U)\subseteq O$. But this is easy: let $U:= [x-\frac{\eps}{h}, x+\frac{\eps}{h}]$. For every $y\in U$ there is a topological  $\tau(x)$--$\tau(y)$~path $\sig''$ contained in $\sig'$  such that $\elll(\sig'')=(y-x)/h$. As the length of any topological path is, by definition, at least the distance of its endvertices, we have $d_A(\tau(x),\tau(y))\leq (y-x)/h$, and thus $\tau(y) \in O$. Since $y$ was arbitrary, we have $\tau(U)\subseteq O$ as required. This proves that $\tau$ is continuous. 
\end{proof}

\section{Existence} \label{secExist}

The existence of a flow of finite energy satisfying \ksl\ in a \lf\ network is a well-known fact, and there are several standard techniques to prove it. Here we will see two such techniques, and point out that they can also be employed when looking for a \cutr\ flow with the above properties. Both these techniques start by showing that in any network $N=(G,r,p,q,I)$, where \g is \lf, \ti\ a \flo{p}{q} $i$ of intensity $I$ that has minimum energy among all such flows; it is then an easy step to show that $i$ satisfies \ksl: if $C$ were a directed cycle with  $\sum_{\are\in \arE(C)} v(\are) >0$, then subtracting from $i$ a constant circular flow around $C$ of sufficiently small intensity yields a flow $i'$ with $W(i')<W(i)$, a contradiction (note that $W(i)<\infty$ since there are flows of finite energy in $N$: just pick a \pth{p}{q}\ $P$ and send a constant flow of intensity $I$ along $P$).

Thus the interesting part is to show the existence of a (\cutr) flow of minimum energy. Our first technique does so using the following well-known fact.

\begin{lemma}[{\cite[Theorem~4.10]{rudin}}]\label{hilb}
If $C$ is a non-empty, closed, convex subset of a Hilbert space, then \ti\ a unique point $y\in C$ of minimum norm among all elements of $C$.
\end{lemma} % *** ---- *** \begin{proof} 	\end{proof}
In order to apply it, let $H$ be the space of all functions $f: E(G) \to \R$ endowed with the inner product $\left<f,g\right>:= \sum_{e\in E(G)} f(e) r(e) g(e)$, and note that the corresponding norm is $||f|| = \sqrt{W(f)}$. It is straightforward to check that this is a Hilbert space, and that the subset $C$ consisting of \cutr\ flows is closed and convex. Thus \Lr{hilb} yields a \cutr\ flow of minimum energy.

Our second technique uses a standard compactness argument. Recall that there is at least one flow in $N$ with finite energy $W$. This $W$ yields, \fe\ edge $e$, an upper bound for the amount of flow $i(e)$ that a flow $i$ of minimum energy, if one exists, can ever send along $e$: this upper bound is $\hat{i}(e):=\sqrt{W/r(e)}$. Now consider \fe\ edge $e$ a topological space $X_e$ homeomorphic to the real interval $[-\hat{i}(e), \hat{i}(e)]$, and let $X:= \Pi_{e} X_e$ be the product of these spaces. Every (\cutr) flow $f$ in $N$ can be represented as a point $p_f$ in $X$: from each component $X_e$ choose the  point of $[-\hat{i}(e), \hat{i}(e)]$ corresponding to $f(e)$ (we are assuming a choice of an orientation for each edge). Let $\seq{f}$ be a sequence of \cutr\ flows in $N$ whose energies converge to the infimum $M$ of the energies of all such flows. By Tychonoff's theorem, $X$ is a compact space; thus, $\seq{f}$ has an accumulation point $w$ in $X$, and it is straightforward to check that the function $i: E(G) \to \R$ corresponding to $w$ is a flow and in fact a \cutr\ one. It is also not hard to see that $W(i)= M$.

We have thus given two proofs of the existence part of our main result:

\begin{theorem}\label{Texist}
Let $N=(G,r,p,q,I)$ be a \lf\ network. Then there is a unique \cutr\ flow $i$ of minimum energy in $N$. This flow satisfies \ksl.
\end{theorem} 
(The uniqueness of $i$ only follows from our first proof.)

Next, we show that the flow $i$ provided by \Tr{Texist} can be obtained as a limit of electrical currents in a sequence of finite networks converging to $N$. The aim of this fact is a better understanding of the concept of \cutr\ flows, but the reader may choose to skip to the next section as we will not make explicit use of this fact later.

Let $G_n, \nin$, be the (finite) subgraph of \g spanned by the vertices at distance at most $n$ from $p$, and let $\dbst{G_n}$ be the graph obtained from $G$ by contracting each component $K$ of $G - G_n$ into a vertex $v_K$ (keeping multiple edges that may result from these contractions). We will, with a slight abuse, use $G_n$ and $\dbst{G_n}$ to also denote the corresponding networks, not just the graphs. For a finite network $H$ we denote by \defi{i(H)} the unique electrical current in $H$. 

\begin{proposition}
Let $N$ be \lf\ network, and let $i$ be the \cutr\ flow in $N$ with minimum energy. Then $i= \lim_n i(\dbst{G_n})$; in particular, the latter limit exists.
\end{proposition} 
% *** ---- *** 
\begin{proof}
Define $\dbst{i_n}: E(\dbst{G_n}) \to \R$ by $ \dbst{i_n}(e) = i(e)$ \fe\ $n$; since $i$ is \cutr, it follows that $\dbst{i_n}$ satisfies \kfl\ for every contracted vertex. Thus $\dbst{i_n}$ is a \flo{p}{q} of value $I$ in $\dbst{G_n}$, and thus $W(i(\dbst{G_n})) \leq W(\dbst{i_n}) \leq W(i)$ for every $n$.
 
Let $g$ be an accumulation point of the sequence $(i(\dbst{G_n}))_\nin$ in the product space $X$ defined as above; here we have to check whether the components $X_e$ are large enough that $X$ can accommodate the elements of the latter sequence, but this is indeed the case since we have shown that $W(i(\dbst{G_n}))\leq W(i)$ \fe\ $n$. The latter inequality also implies 
\labtequ{wgf}{$W(g)\leq W(i)$.}

As in our compactness proof of \Tr{Texist} it is easy to prove that $g$ is a flow in $N$ satisfying \ksl. We moreover claim that $g$ is \cutr. Indeed, for every finite cut $(X,X')$ that does not separate $p$ from $q$ there is an $\nin$ \st\ every edge in $(X,X')$ is contained in $\dbst{G_m}$ \fe\ $m>n$. Since every flow in a finite network is \cutr, $g$ is an accumulation point of flows $i_n$ satisfying $i_n(X,X')=0$; thus $g(X,X')=0$ holds as desired. By \Tr{Texist} and~\eqref{wgf} we obtain $g=i$. In particular, $g=i$ is the only accumulation point of the sequence $(i(\dbst{G_n}))_\nin$, and we may write $i= \lim_n i(\dbst{G_n})$.
\end{proof}

\note{WW thinks that this is true and known for * instead of ** without non-elusiveness}

The last result motivates the following problem.

\begin{conjecture}
If $\lim i(G_n) = \lim i(\dbst{G_n})$ then there is a unique flow in $N$ satisfying \ksl.
\end{conjecture}

\section{The main result: {U}niqueness} \label{suniq}

As mentioned in the introduction, in order to prove the uniqueness of a \cutr\ flow in $N$ satisfying \ksl, we will consider the difference $z$ of two distinct hypothetical flows of this kind; we will not be able to find a cycle in $z$ as in the finite case, but instead we will be able to find a circle. In order to obtain a contradiction, we then need to show that circles must satisfy \eqref{KII} as well. This is an immediate corollary of \Lr{sparse} and the following easy property of real vectors.

\begin{lemma}\label{vectors} 
Let $R,I$ be two infinite-dimensional vectors with positive real values. If $\left<I,R\right> > \left<\unit,R\right>$ then $\left<I^2,R\right> > \left<I,R\right>$. 
\end{lemma} 
(where $\left<\ ,\ \right>$ denotes the usual inner product, $\unit$ denotes the all-ones vector  and $I^2$ denotes the vector whose value at each coordinate is the square of the value of $I$ at that coordinate).
% *** ---- *** 
\begin{proof}
For every component $j$, if there is a gain in this component when going from $\left<\unit,R\right>$ to $\left<I,R\right>$, that is, if $I(j)R(j) \geq R(j)$, then this gain becomes even bigger when going from $\left<I,R\right> $ to $\left<I^2,R\right>$. On the other hand, if there is a loss  in this component when going from $\left<\unit,R\right>$ to $\left<I,R\right>$, then the loss becomes smaller when going from $\left<I,R\right>$ to $\left<I^2,R\right>$, since $I(j)R(j) < R(j)$ in that case. The assertion follows easily from these two observations.
\end{proof}

\begin{corollary}\label{tinfkir}	
Let $N$ be a \lf\ electrical network with resistances $r: E(G) \to \R^+_*$ \st\ $\sum_{e\in E(G)} r(e)<\infty$. Let $i$ be a flow satisfying \ksl\ and $W(i)<\infty$. Then \fe\ directed circle $C$ in \fcg\ \tho\ $\sum_{\are\in \arE(C)} v(\are) =0$; in particular, the latter sum is well-defined.
\end{corollary}
% *** ---- *** 
\begin{proof}
\Lr{sparse} yields a sparse family \ffml{D}{\N}\ of finite directed cycles $D_n$ in \g \st\ $\arE(C) = \sum_n \arE(D_n)$. By \eqref{KII} we have $\sum_{\are \in \arE(D_n)} v(\are)=0$ \fe\ $n$. Thus it suffices to prove that the sum $\sum_n \sum_{\are \in \arE(D_n)} v(\are)$ is absolutely convergent. To prove this, recall that we are assuming that $\sum_{e\in E(G)} r(e)<\infty$ and $\sum_{e\in E(G)} i^2(e) r(e)<\infty$. It follows with \Lr{vectors} that $\sum_{e\in E(G)} |i(e)| r(e)=: u <\infty$. Since the family $D_n$ is sparse, this means that $\sum_n \sum{\are \in \arE(D_n)} |v(\are)| \leq 3 u < \infty$, which proves that the sum is absolutely convergent as desired. Thus \\ $\sum_{\are\in \arE(C)} v(\are) = \sum_n 0 = 0$.
\end{proof}

The following lemma states that any infinite sequence of topological paths between two fixed vertices $x,y$ in $|G|$ has a ``converging'' subsequence whose limit contains a topological $x$-$y$~path. More precisely, given a sequence $E_1,E_2,\ldots$ of sets, let us write 
$$\limf{E_n} := \bigcup_{i\in \N} \bigcap_{j>i} E_j$$ 
for the set of elements eventually in $E_n$.

\begin{lemma}[\cite{hp}] \label{hp}
Let $G$ be a locally finite graph, let $x,y \in V(G)$  and let $(\tau_n)_{n\in \N}$ be a sequence of topological $x$--$y$~paths in $|G|$. Then, there is an infinite subsequence $(\tau_{a_n})_{n\in \N}$ of $(\tau_n)$ and a topological  $x$--$y$~path $\sigma$ in $|G|$ such that $E(\sigma) \subseteq \limf{E(\tau_{a_n})}$. 
Moreover, if no $\tau_n$ traverses any edge more than once then $E(G) \setminus E(\sigma) \subseteq \limf{E(G) \setminus E(\tau_{a_n})}$, no edge is traversed by $\sigma$ more than once, and for every finite subset $F$ of $E(\sigma)$ there is an $m\in\N$ such that the linear ordering of $F$ induced by $\sigma$ coincides with that induced by $\tau_{a_n}$ for every $a_n>m$.
\end{lemma}

We can now finish the proof of our main result.

\begin{proof}[Proof of \Tr{finr}] 
By \Tr{Texist} there exists at least one such flow, so it only remains to prove the uniqueness of $i$. Suppose, to the contrary, there are two \cutr\ flows $i\neq f$ of finite energy in $N$ that both satisfy \ksl, and put $z:= i - f$. 

Since $i\neq f$, there is an oriented edge $\vec{e_0}$ with $z(\vec{e_0}) > 0$. Let $G_1\subseteq G_2 \subseteq G_3\ldots $ be a sequence of finite subgraphs of $G$ such that $\bigcup_n G_n = G$ and $e_0\in E(G_1)$.

Since both $i,f$ are \cutr, so is $z$. Thus $z$ induces a circulation $z_n$ on $\dbst{G_n}$ \fe\ $n$ (\dbst{G_n} is defined as in \Sr{secExist}). Since  $z_n$ a circulation in a finite network, there must exist  an oriented cycle ${C_n}$ in \dbst{G_n} \st\ $\vec{e_0}\in \arE(C_n)$ and $z(\are)>0$ \fe\ $ \are \in E({C_n})$. As $z_n$ is induced from the circulation $z$ of $G$ that, clearly, satisfies \ksl, ${C_n}$ cannot be a cycle of $G$; thus ${C_n}$ contains a contracted vertex of $\dbst{G_n}$. We are now going to make use of the sequence \seq{{C}} in order to construct an oriented circle ${C}$ in \fcg\ with $z(\are)>0$ \fe\ $ \are \in E({C})$. For this, let $P_n$ be the (oriented) path $C - e_0$.

Applying \Lr{hp} to $G$, letting $\tau_n$ be a topological path that traverses $P_n$ in a straight manner, we obtain a topological path \sig, and we can concatenate \sig\ with $e_0$ to obtain a closed topological path $\sig'$, which we may assume traverses $e_0$ in the same direction as the $C_n$ do. By the second sentence of \Lr{hp}, and since no $\tau_n$ visits a vertex more than once, \sig\ also visits no vertex more than once. Moreover, we may assume that every edge $e=xy\in E(\sig')$ is traversed by $\sig$ in the direction in which $z$ flows, in other words, $z(x,y)>0$ if $\sig$ visits $x$ before $y$. Indeed, recall that each $P_n$ traverses its edges in the direction in which $z$ flows. To make sure that $\sig$ traverses its edges in the same direction as the $P_n$ do, we can, before applying \Lr{hp} to obtain \sig, subdivide each edge $e=xy$ of $G$ by a dummy vertex $z$ into two edges $xz,zy$. Now the second sentence of \Lr{hp} implies that if $xz,zy$ are traversed by $\sig$ then they are traversed in the same order as in infinitely many of the $P_n$, which means that $\sig$ traverses $e$ in the same direction as these $P_n$, which is the direction in which $z$ flows along $e$.

By \Lr{oriarc}, there is a circle ${C}$ that shortcuts $\sig'$. This circle clearly violates \Cr{tinfkir}, since for one of its orientations $\vec{C}$ \tho\ $z(\are)>0$ \fe\ $ \are \in E(\vec{C})$. This contradiction completes the proof.

\end{proof}

\Tr{finr} is best possible in the sense that we can not drop any of its requirements. Indeed, to see why the condition of being \cutr\ is necessary, recall that the network of \fig{draynet} has several flows satisfying \ksl\ no matter how we choose its resistances. %let $G$ be a double ray, $p,q$ two adjacent vertices, and suppose that $\sum_{e\in E} r(e)<\infty$. Then there are two distinct $p,q$ flows $f,g$ of value 1 and finite energy in $G$ that both satisfy~\eqref{KII}: let $f$ be the flow satisfying $f(pq)=1$ and $f(e)=0$ for every edge $e\neq pq$, and let $g:= f-1$ (assuming both $f$ and $g$ direct each edge from left to right). 

To see that the condition $W(i)<\infty$ is necessary, consider the network $N$ of \fig{harmonic}. We will construct a non-trivial non-elusive \flo{p}{q}\ $f$ of intensity 0 in $N$, i.e.\ a circulation, that satisfies \ksl. Let $f(p,q)=1$, and let $f$ also send a flow of value 1 along the two edges incident with $p,q$, so that~\eqref{kI} is satisfied at both $p,q$. We can now assign a flow $f(e)$ to the perpendicular edge $e$ forming a 4-cycle $C$ with those three edges so that $C$ complies with~\eqref{KII}. Then, we can assign a flow to each of the two edges incident with $e$ so that the endvertices of $e$ comply with~\eqref{kI}. Continuing like this, we obtain a function $f$  that satisfies both Kirchhoff laws and is non-elusive.
\showFig{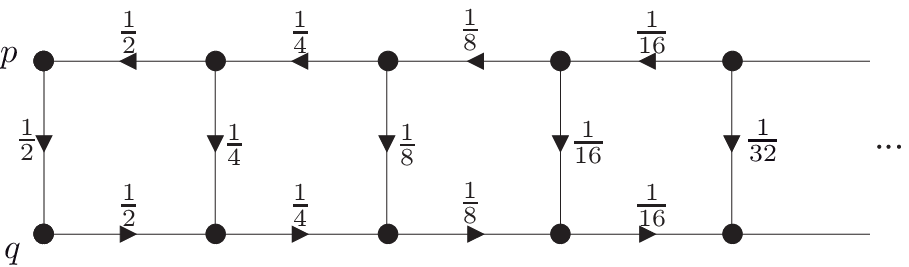}{A network that has a non-constant non-elusive circulation satisfying \ksl. The numbers on the edges denote their resistances.}

In view of \Tr{finl} it is tempting to conjecture that the requirement $\sum r(e)< \infty$ in \Tr{finr} can be replaced by the weaker requirement $\ltpf{r} \eqT \fcg$. However, in \cite{kirch2} we show a counterexample to this conjecture.

% ***************************
\comment{  
	The second definition uses the concept of energy: the \defi{energy of the flow $f$} is defined by $W(f):= \sum_{e\in E(G)} f^2(e) r(e)$. ($W(f)$ is typically called ``energy'' in mathematics, but in physics it is called ``power''. In physics the term ``energy of a current'' is used for another concept ... The following theorem is well known.

	\begin{theorem}[]
	Let $W=W(I)$ be the infimum of $W(f)$ taken over all $\flos{p}{q}$ $f$ of value $I$ in $G$. Then, there is a unique $\flo{p}{q}$ $i$ in $G$ satisfying $W(i)=W$. This flow $i$ satisfies \eqref{KII}.
\end{theorem}

	This theorem enables us to state the following definition, which is equivalent to \Dr{defClas}.

	\begin{definition}\label{defTho}
	the \defi{electrical current} $i$ of value $I\in \R$ from $p$ to $q$ in $G$ is the (unique) $\flo{p}{q}$ of value $I\in \R$ having minimal energy.
\end{definition}

}
% ***************************

\section{Finite total conductance} \label{scond}

It is known that networks of finite total conductance have unique electrical currents. One way to prove this is to notice that in this case the corresponding random walk is recurrent \cite[Lemma 4.2]{WoessBook09}, which combined with Lyons' theorem (see \cite{LyoSim} or \cite{LyonsBook}), and a little additional effort, implies uniqueness of electrical currents \cite{SchleInf}. See \Sr{stoch} or \cite{LyonsBook,WoessBook09} for more on the relationship between electrical networks and random walks. In this section we give a non-probabilistic proof of this fact for the sake of completeness.

\begin{theorem}\label{finc}
Let $N=(G,r,p,q,I)$ be countable network with $\sum_{e\in E(G)} 1/r(e) < \infty$. Then there is a unique flow of finite energy satisfying \ksl. This flow is \cutr.
\end{theorem} 
% *** ---- *** 
\begin{proof}
Suppose there are distinct flows $f,g$ in $N$ both having finite energy and satisfying \ksl\, and consider their difference $z:= f-g$. It is not hard to check that the circulation $z$ has finite energy too, and it clearly also satisfies \ksl.

Pick an edge $f$ \st\ $z(f)> 0$. Since $\sum_{e\in E(G)} 1/r(e) < \infty$ and $\sum_{e\in E(G)} z^2(e)r(e) < \infty$, \ti\ a finite set of edges $F\subset E(G)$ \st\ 
\labtequ{F}{$\sum_{e\in E(G) \sm F} |z(e)| < |z(f)|$.}
Let $U$ be the set of vertices incident with an edge in $F$, and let $G^*$ be the graph obtained from $G$ after identifying all vertices in $V(G) \sm U$ into a single vertex $v^*$, keeping parallel edges if any arise. Note that $z$, considered as a function from $E(G)=E(G^*)$ to \R, is also a circulation in $G^*$; the sum $\sum \{z(e) \mid \text{e is incident with $v^*$} \}$ is well-defined even if $v^*$ has infinite degree by \eqref{F}, and it equals zero since the set $U$ is finite and every vertex in it satisfies \eqref{kI}. We may assume that the graph $G^*$ has only finitely many edges, for if $x,y\in V(G^*)$ are joined by an infinite set of edges $B$, then we may replace all these edges by a single edge carrying the flow $\sum_{e\in B} z(e)$.

We claim that $z$ must traverse a cycle in $G^*$ that does not visit $v^*$. Indeed, since $z$ is a circulation in the finite graph $G^*$, it must traverse some cycle $C_0$. If $C_0$ visits $v^*$, then we can subtract from $z$ a constant circulation along $C_0$ to obtain a new circulation $z_1$ that does not traverse $C_0$. By \eqref{F} we have $z_1(f)>0$ since no edges incident with $v^*$ lie in $E(G)\sm F$. If $z_1$ still traverses a cycle $C_1$ visiting $v^*$ we can subtract it to obtain $z_2$, and so on. Continuing like this we can, after finitely many steps, reach a circulation $z_k$ that traverses no cycle incident with $v^*$ and satisfies $z_k(f)>0$. But then, $z_k$ must traverse a cycle $C$ not incident with $v^*$, and by the construction of $z_k$ there follows that $z$ also traverses $C$. But $C$ is also a cycle in the original graph $G$, which contradicts the fact that $z$ satisfies \ksl.

Thus we have proved the uniqueness of a flow $i$ of finite energy satisfying \ksl. To prove that $i$ is \cutr, it would now suffice to show the existence of a \cutr\ flow $f$ in $N$ having finite energy and satisfying \ksl\, since uniqueness implies $f=i$. We already proved the existence of such a flow $f$ in \Tr{Texist} in the case that $G$ is \lf. The interested reader will be able to check that our second proof of \Tr{Texist} also works in the \nlf\ case if $\sum_{e\in E(G)} 1/r(e) < \infty$ ---although not necessarily otherwise, see \Sr{NLF}.
\end{proof}

\section{Non-locally-finite networks} \label{NLF}

In this section we discuss how the results of this paper behave \wrt\ \nlfg s.

Let us start with an example. The graph \g of \fig{fan} consists of two edges $e=pu, f=vq$ and an infinite family \fml{P}\ of independent \pths{u}{v}\ of length two. Let us suppose that all edges in the $P_n$ have the same resistance, although this makes little difference. It is straightforward to check that the corresponding network does not have a flow of minimum energy (with a fixed intensity $I$), and in fact no flow that satisfies \ksl. 

\showFig{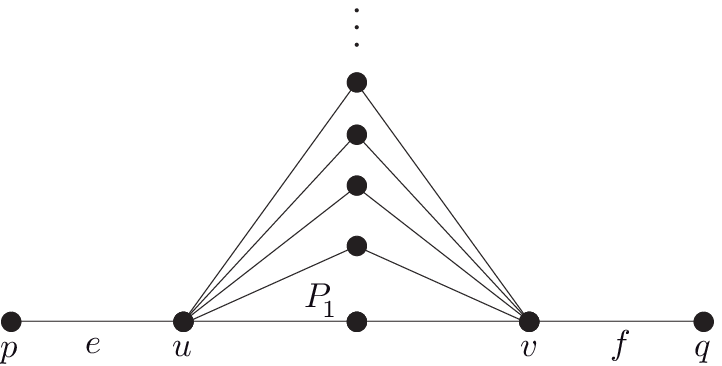}{A simple example showing that \nlf\ networks do not necessarily have flows satisfying \ksl.}

This seems to suggest that the results of this paper cannot be extended to \nlfg s, but nevertheless there is a way around this: call a function $f: V^2 \to \R$ a \defi{\refl} from $p$ to $q$ if it satisfies $f(x,y)=-f(y,x)$ \fe\ $x,y$, $f(x,y)=0$ if $xy\not\in E(G)$, and moreover it is \cutr\ (defined as in \Sr{sdefs}). Note that if \g is \lf\ then every \refl\ is a \cutr\ flow, since applying the non-elusiveness condition to the set of edges incident with a given vertex $x$ yields \kfl\ for $x$. It is worth mentioning that this definition effectively imposes \kfl\ to the points of \eti\ rather than the vertices of \G.

The interested reader will be able to show that \Tr{finr} extends to \nlfg s to assert the existence and uniqueness of a \refl\ of finite energy if the total resistance is finite (and thus the graph countable).
Indeed, all intermediate results we used are proved (either here or in the respective sources) for countable graphs as well, and the rest of the proof of \Tr{finl} easily extends to countable graphs.

\section{Relationship to stochastic processes and the discrete Dirichlet Problem} \label{stoch}

There is a well studied and fruitful relationship between electrical networks and random walks. Every electrical network gives rise to a random walk, in which the transition probabilities are proportional to the conductances of the corresponding edges and, conversely, every reversible irreducible Markov chain can be represented by an electrical network \cite{LyonsBook,WoessBook09}. 

It turns out that in a finite network the probability that random walk starting at a vertex $x$ reaches a fixed vertex $p$ before having visited another fixed vertex $q$ equals the potential of vertex $x$ induced by a voltage source imposing potential~1 at $p$ and potential 0 at $q$. It is natural to ask if this relationship carries over to the networks of the current paper, in other words, if the unique flow provided by \Tr{finr} coincides with the flow induced by the corresponding random walk on the same graph. This is however not the case: in the graph of \fig{draynet}, if the sum of the resistances is finite then the probability that random walk starting at a vertex $x\neq p$ to the left of $p$ visits $p$ before having visited $q$ is less than 1, since such a random walk is transient and so the probability that $p$ is never visited is positive. It follows that the corresponding flow fails to be \cutr. 

It would be interesting though to try to define a continuous random process, i.e.\ a brownian motion, on $\ltpf{r}$ in which the particle continues moving after reaching a boundary point of $\ltpf{r}$; for such a process it might be the case that the corresponding flow is indeed \cutr. To support this idea, we now show that given a network $N$ and a flow $f$ of finite energy in $N$ satisfying \ksl, it is possible to extend the potential function $P: V(G) \to \R$ from the vertices of \g to all of $\ltpf{r}$ in a natural way. Recall that given a flow $i$ satisfying \ksl, one usually defines $P$ by fixing $P(q)=0$ and then letting $P(u)= \sum_{\are\in \arE(R)} v(\are)$, where $R$ is any path from $q$ to $u$; \ksl\ implies that $P(u)$ does not depend on the choice of $R$ but only on its endpoints. The following lemma allows us to extend $P$ to the boundary $\partial^r G$ of $\ltpf{r}$ whenever we have a flow of finite energy satisfying \ksl. Note that for every such boundary point $x$ there is a ray of finite total resistance converging to $x$ in \ltpf{r}, see \cite[Lemma 4.1]{ltop}.

%function $P:V \to \R$ continuously to $\ltpf{r}$.

\comment{
\begin{problem} \label{infKirc}
Let $G$ be a \lf\ electrical network with resistances $r: E(G) \to \R^+_*$. Let $i$ be a flow satisfying \eqref{KII} with $W(i)<\infty$. Then every circle $C$ in $\ltpf{c}$ such that $E(C)$ is dense in $C$ and $\ell(C)<\infty$ satisfies~\eqref{KIIp} with respect to $i$. 
\end{problem}
}

\begin{lemma}\label{indep}	
Let $f$ be a flow of finite energy satisfying \ksl\ in a \lf\ network $N=(G,r,p,q,I)$, let $x\in \partial^r G$, and let $R$ be a ray of finite total resistance starting at $q$ and converging to $x$. Then $\sum_{\are\in \arE(R)} f(\are) r(\are)$ is well defined and does not depend on $R$.
\end{lemma} 
% *** ---- *** 
\begin{proof}[Proof (sketch).]
Well-definedness follows easily by \Lr{vectors}. To show independence from $R$, let $R,S$ be two such rays. Since, easily, $R$ and $S$ belong to the same end of $G$, it is possible to find an infinite sequence \seq{P}\ of disjoint $R,S$ paths \st\ $\sum_{e\in \bigcup_i P_i} r(e) < \infty$, the endpoint of $P_i$ in $R$ comes after the endpoint of $P_{i-1}$ \fe\ $i$, and similarly for the endpoints in $S$. The sum of potential differences $\sum_{e\in \bigcup_i P_i} f(e) r(e)$ is absolutely convergent by \Lr{vectors}. Adding potential differences along all finite cycles formed by $R,S$ and a pair of subsequent $P_i$, and using the fact that these cycles satisfy~\eqref{KII}, we obtain $\sum_{\are\in \arE(R)} f(\are) r(\are) - \sum_{\are\in \arE(S)} f(\are) r(\are) =0$.
\end{proof}

Thus we can extend $P$ to the boundary $\partial^r G$ by letting $P(x)= \sum_{\are\in \arE(R)} f(\are) r(\are)$ for some such ray $R$. The interested reader will be able to check that the function $P: \ltpf{r} \to \R$ defined this way is continuous provided $f$ has finite energy.

Conversely, one may ask whether given any continuous potential function $P: \partial^r G \to \R$ on the boundary of \ltpf{r}\ \ti\ a flow in \g inducing $P$. This is an instance of the well-known discrete Dirichlet Problem. Let us state it more precisely:

\begin{conjecture} \label{condir}
 Let \g be a \lfg\ and fix a function $r: E(G)\to \R^+$ \st\ \ltpf{r}\ is compact. Then, \fe\ continuous function $P: \partial^r G \to \R$ \ti\ a circulation in \g that satisfies \ksl\ and induces $P$ in the above sense.
\end{conjecture}

\hyphenation{com-pac-ti-fi-cation}
A typical instance of the discrete Dirichlet Problem presupposes a \lfg\ \g with a fixed compactification, usually the \FC\ $|G|$ or the hyperbolic compactification, and asks whether every continuous function on the boundary can be induced by a circulation in \G; see for example \cite{BeSchrHar,KaWoDir,WoeDir}. In most cases, all edges of \g are considered with resistance~1. In comparison, \Cnr{condir} allows more flexibility as far as the choice of resistances and the choice of a compactification is concerned, but demands that these two choices are compatible with each other. 

An interesting aspect of \Cnr{condir} is the fact that every compact metric space is isometric to $\partial^r G$ for some appropriate choice of \g and $r$, see \cite[Section~4]{ltop}.

\Cnr{condir} has been proved by Carlson\cite{CarBou} for the special case that is perhaps most interesting from the point of view of the current paper: the case when $\sum_{e\in E} r(e)<\infty$.

\newcommand{\acknowledgements}{\section*{Acknowledgements}}

\acknowledgements{I am very grateful to Wolfgang Woess for several discussions on the topic, in particular for drawing my attention to the applicability of \Lr{hilb} in this context, as well as for acquainting me with the proof sketch at the beginning of \Sr{scond}}. I am also very grateful to Reinhard Diestel for motivating me to study electrical networks and for valuable discussions.

\note{

\begin{problem}
If \ltp\ is compact, is the \cutr\ electrical flow in \g a sum of flows along topological paths and circles in \ltp? This could help generalise \Tr{finr} to the Floyd boundary.
\end{problem}
}

%\begin{abstract}
%\end{abstract}

\bibliographystyle{plain}
\bibliography{collective}
\end{document}